\documentclass[11pt,twoside,reqno]{amsart}

\usepackage{amssymb, latexsym}
\usepackage[italian,english]{babel}
\usepackage{amsmath,amsfonts,amsthm}
\usepackage{paralist}
\usepackage[top=2.50cm, bottom=2.50cm, left=2.50cm, right=2.50cm]{geometry}

\usepackage{color}
\numberwithin{equation}{section}
\usepackage{xspace}
\usepackage[bookmarksnumbered,colorlinks]{hyperref}
\usepackage{graphics}
\def\bb#1\eb{\textcolor{blue}
{#1}} %
\def\br#1\er{\textcolor{red}
{#1}} %
\def\bv#1\ev{\textcolor{green}
{#1}} %
\def\bc#1\ec{\textcolor{cyan}
{#1}} %

\usepackage{graphics}

\def\Xint#1{\mathchoice
  {\XXint\displaystyle\textstyle{#1}}%
  {\XXint\textstyle\scriptstyle{#1}}%
  {\XXint\scriptstyle\scriptscriptstyle{#1}}%
  {\XXint\scriptscriptstyle\scriptscriptstyle{#1}}%
  \!\int}
\def\XXint#1#2#3{{\setbox0=\hbox{$#1{#2#3}{\int}$}
  \vcenter{\hbox{$#2#3$}}\kern-.5\wd0}}
\def\-int{\Xint -}

\newcommand{\e}{\varepsilon}
\newcommand{\R}{\mathbb{R}}
\newcommand{\N}{\mathbb{N}}

\DeclareMathOperator{\dive}{div}
\DeclareMathOperator{\supp}{supp}

\newtheorem{prop}{Proposition}[section]
\newtheorem{lem}{Lemma}[section]
\newtheorem{thm}{Theorem}[section]

\newtheorem{remark}{Remark}[section]

\begin{document}
\title[Fractional Schr\"odinger equation]{Multiplicity of positive solutions for a class of fractional Schr\"odinger equations via penalization method}

%\author[V. Ambrosio]{Vincenzo Ambrosio}
%\address{Dipartimento di Matematica e Applicazioni ``R. Caccioppoli''\\
%         Universit\`{a} degli Studi di Napoli Federico II\\
%         via Cinthia, 80126 Napoli, Italy}
%\email{vincenzo.ambrosio2@unina.it}
\author[V. Ambrosio]{Vincenzo Ambrosio}
\address{Dipartimento di Scienze Pure e Applicate (DiSPeA),
Universit\`a degli Studi di Urbino 'Carlo Bo'
Piazza della Repubblica, 13
61029 Urbino (Pesaro e Urbino, Italy)}
\email{vincenzo.ambrosio2@unina.it}

\keywords{Fractional Laplacian; penalization method; multiplicity of solutions; Nehari manifold; supercritical problems}
\subjclass[2010]{35A15, 35J60, 35R11, 45G05}

\date{}

\date{}

\begin{abstract}
By using the penalization method and the Ljusternik-Schnirelmann theory, we investigate the multiplicity of positive solutions of the following fractional Schr\"odinger equation
$$
\e^{2s}(-\Delta)^{s} u + V(x)u = f(u) \mbox{ in } \R^{N}
$$
where $\varepsilon>0$ is a parameter, $s\in (0, 1)$, $N>2s$, $(-\Delta)^{s}$ is the fractional Laplacian, $V$ is a positive continuous potential with local minimum, and $f$ is a superlinear function with subcritical growth.
We also obtain a multiplicity result when $f(u)=|u|^{q-2}u+\lambda |u|^{r-2}u$ with $2<q<2^{*}_{s}\leq r$ and $\lambda>0$, 
by combining a truncation argument and a Moser-type iteration.
\end{abstract}

\maketitle

\section{Introduction}

\noindent

In this paper we deal with the following nonlinear fractional Schr\"odinger equation
\begin{equation}\label{P}
\e^{2s}(-\Delta)^{s} u + V(x)u =  f(u) \mbox{ in } \R^{N} 
\end{equation}
where $\e>0$ is a parameter, $s\in (0,1)$ and $N>2 s$. Here, the potential $V: \R^{N}\rightarrow \R$ is continuous and verifies the following hypotheses:
\begin{compactenum}[$(V_1)$]
\item $\inf_{x\in \R^{N}} V(x)=V_{0}>0$;
\item there exists a bounded set $\Lambda\subset \R^{N}$ such that
$$
V_{0}<\min_{x\in \partial \Lambda} V(x).
$$
\end{compactenum}
Concerning the nonlinearity $f: \R\rightarrow \R$, we assume that $f(t)=0$ for $t<0$, and that it is a  continuous function satisfying the following conditions: 
\begin{compactenum}[($f_1$)]
\item $\displaystyle{\lim_{t\rightarrow 0} \frac{f(t)}{t}=0}$;
\item there exists $q\in (2, 2^{*}_{s})$, where $2^{*}_{s}=\frac{2N}{N-2s}$, such that $\displaystyle{\lim_{t\rightarrow \infty} \frac{f(t)}{t^{q-1}}=0}$;
\item there exists $\theta>2$ such that
\begin{equation*}
0< \theta F(t) = \theta \int_{0}^{t} f(\tau) \, d\tau \leq f(t) \, t \, \mbox{ for all } t> 0;
\end{equation*}
\item The map $\displaystyle{t \mapsto \frac{f(t)}{t}}$ is increasing for every $t>0$.
\end{compactenum}
\noindent

The nonlocal operator $(-\Delta)^{s}$ appearing in (\ref{P}), is the so-called fractional Laplacian, that, up to a positive constant, is defined by
$$
(-\Delta)^{s}u(x):={\rm P.V.}\int_{\R^{N}} \frac{u(x)-u(y)}{|x-y|^{N+2s}} dy,
$$
for every $u:\R^{N}\rightarrow \R$ sufficiently smooth; see for instance \cite{DPV}.\\
We recall that the study of (\ref{P}) is motivated by the search of standing wave solutions for the time-dependent fractional Schr\"odinger equation
\begin{equation*}
\imath \frac{\partial \psi}{\partial t}=(-\Delta)^{s} \psi+V(x)\psi -g(x, |\psi|) \quad (t, x)\in \R\times \R^{N},
\end{equation*}
which plays a fundamental role in the fractional quantum mechanic \cite{Laskin1, Laskin2}. \\
From the point of view of nonlinear analysis, we mention the works \cite{A1, A2, A3, A4, CW, DPPV, FFV, FQT, FLS, MBR, Secchi1} where several existence and multiplicity results for fractional Schr\"odinger equations have been obtained via different variational methods.
More in general, the study of nonlinear elliptic equations involving nonlocal and fractional operators has gained tremendous popularity during the last decade, because of intriguing structure of these operators and their application in many areas of research such as optimization, finance, phase transition phenomena, minimals surfaces, game theory, population dynamics. For more details and applications on this topic we refer the interested reader to \cite{DPV, MBRS}.
\noindent

Recently, the analysis of concentration phenomenon of solutions for the nonlinear fractional Schr\"odinger equation (\ref{P}) has attracted the attention from many mathematicians.\\
Davila et al. \cite{DDPW} proved via Liapunov-Schmidt reduction method, that (\ref{P}) has multi-peak  solutions, under the assumptions that $f(u)=u^{p}$ with $p\in (1, 2^{*}_{s}-1)$, and the potential $V$ verifies the following conditions
$$
V\in C^{1, \alpha}(\R^{N})\cap L^{\infty}(\R^{N}) \mbox{ and } \inf_{x\in \R^{N}} V(x)>0.
$$
Shang and Zhang \cite{SZ} studied the following class of fractional Schr\"odinger equations 
\begin{align*}
\varepsilon^{2s}(-\Delta)^{s} u+V(x)u=K(x)|u|^{p-1}u \mbox{ in } \R^{N} 
\end{align*}
where $V$ and $K$ are positive smooth functions, showing the existence and multiplicity of solutions which concentrate near some critical points of $\Gamma(x)=[V(x)]^{\frac{p-1}{p+1}-\frac{N}{2s}}[K(x)]^{-\frac{2}{p-1}}$ by applying perturbative variational method.\\
By means of the Lyusternik-Shnirelmann and Morse theories, Figueiredo and Siciliano \cite{FS} proved a multiplicity result for (\ref{P}), with $f\in C^{1}$ and satisfying the hypotheses $(f_1)$-$(f_4)$.\\
Alves and Miyagaki \cite{AM} dealt with the existence of a positive solution to (\ref{P}) by combining the penalization method developed in \cite{DF} and the Caffarelli-Silvestre extension technique \cite{CS}.\\
He and Zou \cite{HZ} investigated the existence and the concentration of positive solutions for a class of fractional Schr\"odinger equations involving the critical Sobolev exponent.\\
Motivated by the above papers, the aim of this work is to study the multiplicity and the concentration of positive solutions of (\ref{P}), involving a continuous nonlinearity satisfying the assumptions $(f_1)$-$(f_4)$. In particular, we are interested in relating the number of positive solutions of (\ref{P}) with the topology of the set $M=\{x\in \Lambda: V(x)=V_{0}\}$. We recall that if $Y$ is a given closed set of a topological space $X$, we denote by $cat_{X}(Y)$ the Ljusternik-Schnirelmann category of $Y$ in $X$, that is the least number of closed and contractible sets in $X$ which cover $Y$.\\
The first main result of this paper is the following:
\begin{thm}\label{thmf}
Suppose that $V$ verifies $(V_1)$-$(V_2)$ and $f$ satisfies $(f_1)$-$(f_4)$. Then, for any $\delta>0$ such that $M_{\delta}=\{x\in \R^{N}: dist(x, M)\leq \delta\}\subset \Lambda$, there exists $\e_{\delta}>0$ such that, for any $\e\in (0, \e_{\delta})$, the problem  \eqref{P} has at least $cat_{M_{\delta}}(M)$ solutions. Moreover, if $u_{\e}$ denotes one of these positive solutions and $x_{\e}\in \R^{N}$ its global maximum, then 
$$
\lim_{\e\rightarrow 0} V(x_{\e})=V_{0}.
$$
\end{thm}

\noindent
The proof of Theorem \ref{thmf} is obtained by using critical point theory and following some ideas used in \cite{AF, DF}. Since we don't have any information about the behavior of potential $V$ at infinity, we use the penalization method introduced by del Pino and Felmer in \cite{DF}. \\
Firstly, by using the change of variable $u(x)\mapsto u(\e x)$ we can see that the problem (\ref{P}) is equivalent to the following one
\begin{equation}\label{R}
(-\Delta)^{s} u + V(\e x)u =  f(u) \mbox{ in } \R^{N}. 
\end{equation}
Fix $K>\frac{2\theta}{\theta-2}>2$ and $a>0$ such that $\frac{f(a)}{a}=\frac{V_{0}}{k}$, and we introduce the functions
$$
\tilde{f}(t):=
\begin{cases}
f(t)& \text{ if $t \leq a$} \\
\frac{V_{0}}{k} t   & \text{ if $t >a$}.
\end{cases}
$$ 
and
$$
g(x, t)=\chi_{\Lambda}(x)f(t)+(1-\chi_{\Lambda}(x))\tilde{f}(t),
$$
where $\chi_{\Lambda}$ is the characteristic function on $\Lambda$, and  we write $G(x, t)=\int_{0}^{t} g(x, \tau)\, d\tau$.\\
Let us note that from the assumptions $(f_1)$-$(f_3)$, $g$ satisfies the following properties:
\begin{compactenum}[($g_1$)]
\item $\displaystyle{\lim_{t\rightarrow 0} \frac{g(x, t)}{t}=0}$ uniformly in $x\in \R^{N}$;
\item $g(x, t)\leq f(t)$ for any $x\in \R^{N}$ and $t>0$;
\item $0< \theta G(x, t)\leq g(x, t)t$ for any $x\in \Lambda$ and $t>0$;
\item $0\leq 2 G(x, t)\leq g(x, t)t\leq \frac{V_{0}}{K}t^{2}$ for any $x\in \R^{N}\setminus \Lambda$ and $t>0$.
\end{compactenum}
Thus, we consider the following auxiliary problem 
\begin{equation}\label{Pe}
(-\Delta)^{s} u + V(\e x)u =  g(\e x, u) \mbox{ in } \R^{N} 
\end{equation}
and we note that if $u$ is a solution of (\ref{Pe}) such that 
\begin{equation}\label{ue}
u(x)<a \mbox{ for all } x\in \R^{N}\setminus \Lambda_{\e},
\end{equation}
where $\Lambda_{\e}:=\{x\in \R^{N}: \e x\in \Lambda\}$, then $u$ solves (\ref{R}), in view of the definition of $g$.\\
It is clear that, weak solutions to (\ref{Pe}) are critical points of the Euler-Lagrange functional
$$
J_{\e}(u)=\frac{1}{2}\int_{\R^{N}}|(-\Delta)^{\frac{s}{2}}u|^{2}+V(\e x) u^{2}\, dx-\int_{\R^{N}} G(\e x, u)\, dx
$$
defined on the fractional space
$$
H^{s}_{\e}=\{u\in \mathcal{D}^{s, 2}(\R^{N}): \int_{\R^{N}} V(\e x) u^{2}\, dx<\infty\}.
$$
Hence, as in \cite{AF}, it seems natural to work on  the Nehari manifold $\mathcal{N}_{\e}$ associated to $J_{\e}$.
Anyway, $f$ is only continuous, so $\mathcal{N}_{\e}$ is not differentiable, and we cannot adapt in our framework the techniques developed in \cite{AF}, to deduce a multiplicity result for (\ref{Pe}). To circumvent this difficulty, we will exploit some abstract category results obtained in \cite{SW}. 
After that, we will make use of the fact that solutions $u_{\e}$ of (\ref{Pe}) have a polynomial decay at infinity uniformly in $\e$ \cite{AM}, to show that for all $\e>0$ small enough, these functions $u_{\e}$ verify (\ref{ue}). \\
We would like to observe that our result complement the result in \cite{AM}, in the sense that now we consider the question related to the multiplicity. Moreover, Theorem \ref{thmf} can be seen as the fractional analogue of the multiplicity result obtained in \cite{AF} (see Theorem $1.1$ in \cite{AF}), but assuming  that the nonlinearity $f$ is only continuous, and not $C^{1}$.
\smallskip

\noindent
In the second part of the paper, we consider a supercritical version of problem (\ref{P}). More precisely, we deal with the following parametric problem  
\begin{equation}\label{Pcritical}
\e^{2s}(-\Delta)^{s} u + V(x)u =  |u|^{q-2}u+\lambda |u|^{r-2}u \mbox{ in } \R^{N} 
\end{equation}
where $\e>0$, $\lambda>0$, and $2<q<2^{*}_{s}\leq r$. We recall that when $r=2^{*}_{s}$ and $\lambda=1$, the multiplicity for problem (\ref{Pcritical}) has been studied in \cite{HZ}.

Our second main result can be stated as follows:
\begin{thm}\label{thmf2}
Suppose that $V$ verifies $(V_1)$-$(V_2)$. Then there exists $\lambda_{0}>0$ with the following property: for any $\lambda\in (0, \lambda_{0})$ and $\delta>0$ given, there exists $\e_{\lambda, \delta}>0$ such that, for any $\e\in (0, \e_{\lambda, \delta})$, the problem  \eqref{Pcritical} has at least $cat_{M_{\delta}}(M)$ solutions.
Moreover, if $u_{\e}$ denotes one of these positive solutions and $x_{\e}\in \R^{N}$ its global maximum, then 
$$
\lim_{\e\rightarrow 0} V(x_{\e})=V_{0}.
$$
\end{thm}

\noindent
The proof of the above result is based on some arguments developed in \cite{CY, FF, Rab}. We first truncate in a suitable way the nonlinearity on the left hand side of \eqref{Pcritical}, in order to deal with a new subcritical problem.
Taking into account Theorem \ref{thmf}, we know that a multiplicity result for this truncated problem holds. Then, we deduce a priori bounds for these solutions, and by using a Moser iteration technique \cite{Moser}, we are able to show that, for $\lambda>0$ sufficiently small, the solutions of the truncated problem also satisfy the original problem \eqref{Pcritical}. \\
As far as we know, in the current literature do not appear multiplicity results for supercritical fractional problems via Ljusternik-Schnirelmann theory, so all results presented here are new. 

\smallskip

\noindent
The body of the paper is the following. In Section $2$ we collect some notations and we give some technical lemmas. The Section $3$ is devoted to prove the existence of multiple solutions to (\ref{P}). In Section $4$ we give the proof of Theorem \ref{thmf2}.

\section{Preliminaries}

In this section we fix the notations and we prove some useful lemmas.
For any $s\in (0,1)$ we define $\mathcal{D}^{s, 2}(\R^{N})$ as the completion of $C^{\infty}_{0}(\R^{N})$ with respect to
$$
[u]^{2}=\iint_{\R^{2N}} \frac{|u(x)-u(y)|^{2}}{|x-y|^{N+2s}} \, dx \, dy =\|(-\Delta)^{\frac{s}{2}} u\|^{2}_{L^{2}(\R^{N})},
$$
that is
$$
\mathcal{D}^{s, 2}(\R^{N})=\left\{u\in L^{2^{*}_{s}}(\R^{N}): [u]_{H^{s}(\R^{N})}<\infty\right\}.
$$
Let us introduce the fractional Sobolev space
$$
H^{s}(\R^{N})= \left\{u\in L^{2}(\R^{N}) : \frac{|u(x)-u(y)|}{|x-y|^{\frac{N+2s}{2}}} \in L^{2}(\R^{2N}) \right \}
$$
endowed with the natural norm 
$$
\|u\|_{H^{s}(\R^{N})} = \sqrt{[u]^{2} + \|u\|_{L^{2}(\R^{N})}^{2}}.
$$

\noindent
For the convenience of the reader, we recall the following enbeddings:
\begin{thm}\cite{DPV}\label{Sembedding}
Let $s\in (0,1)$ and $N>2s$. Then there exists a sharp constant $S_{*}=S(N, s)>0$
such that for any $u\in H^{s}(\R^{N})$
\begin{equation}\label{FSI}
\|u\|^{2}_{L^{2^{*}_{s}}(\R^{N})} \leq S_{*} [u]^{2}. 
\end{equation}
Moreover $H^{s}(\R^{N})$ is continuously embedded in $L^{q}(\R^{N})$ for any $q\in [2, 2^{*}_{s}]$ and compactly in $L^{q}_{loc}(\R^{N})$ for any $q\in [2, 2^{*}_{s})$. 
\end{thm}

\noindent
We also recall the following Lions-compactness lemma.
\begin{lem}\cite{FQT}\label{lions lemma}
Let $N>2s$. If $(u_{n})$ is a bounded sequence in $H^{s}(\R^{N})$ and if
$$
\lim_{n \rightarrow \infty} \sup_{y\in \R^{N}} \int_{B_{R}(y)} |u_{n}|^{2} dx=0
$$
where $R>0$,
then $u_{n}\rightarrow 0$ in $L^{t}(\R^{N})$ for all $t\in (2, 2^{*}_{s})$.
\end{lem}

\noindent
For any $\e>0$, we denote by $H^{s}_{\e}$ the completion of $C^{\infty}_{0}(\R^{N})$ with respect to the norm
$$
\|u\|^{2}_{H^{s}_{\e}}=\iint_{\R^{2N}} \frac{|u(x)-u(y)|^{2}}{|x-y|^{N+2s}}\, dx dy+\int_{\R^{N}} V(\e x) u^{2}(x)\, dx.
$$
It is clear that $H^{s}_{\e}$ is a Hilbert space with respect to the inner product 
$$
(u, v)_{H^{s}_{\e}}=\iint_{\R^{2N}} \frac{(u(x)-u(y))}{|x-y|^{N+2s}}(v(x)-v(y))\, dx dy+\int_{\R^{N}} V(\e x) u v\, dx.
$$
When $\e=0$, we set $H^{s}_{0}=H^{s}(\R^{N})$ and $\|u\|^{2}_{H^{s}_{0}}=[u]^{2}+\int_{\R^{N}} V_{0} u^{2}\, dx$. 
To study (\ref{R}), we seek critical points of the following $C^{1}$-functional
$$
J_{\e}(u)=\frac{1}{2}\|u\|^{2}_{H^{s}_{\e}}-\int_{\R^{N}} G(\e x, u)\, dx
$$
defined for any $u\in H^{s}_{\e}$. 
We also define the following autonomous functional
$$
J_{0}(u)=\frac{1}{2}\|u\|^{2}_{H^{s}_{0}}-\int_{\R^{N}} F(u)\, dx 
$$
for any $u\in H^{s}(\R^{N})$.\\
Arguing as in \cite{AM}, it is easy to show that $J_{\e}$ satisfies the assumptions of the mountain pass theorem \cite{AR}. More precisely, we have
\begin{lem}\label{MPG}
$J_{\e}$ has a mountain pass geometry, that is
\begin{compactenum}[(i)]
\item $J_{\e}(0)=0$;
\item there exists $\alpha, \rho>0$ such that $J_{\e}(u)\geq \alpha$ for any $u\in H^{s}_{\e}$ such that $\|u\|_{H^{s}_{\e}}=\rho$;
\item there exists $e\in H^{s}_{\e}$ with $\|e\|_{H^{s}_{\e}}>\rho$ such that $J_{\e}(e)<0$.
\end{compactenum}
\end{lem}

\begin{lem}\label{PSc}
Let $c\in \R$. Then, $J_{\e}$ satisfies the Palais-Smale condition at the level $c$.
\end{lem}

\noindent
Taking into account Lemma \ref{MPG} and Lemma \ref{PSc}, we can define the mountain pass level
$$
c_{\e}=\inf_{\gamma\in \Gamma_{\e}} \max_{t\in [0, 1]} J_{\e}(\gamma(t))
$$
where
$$
\Gamma_{\e}=\{\gamma\in C([0, 1], H^{s}_{\e}): \gamma(0)=0 \mbox{ and } \gamma(1)=e\},
$$
and to deduce that there exists $u_{\e}\in H^{s}_{\e}\setminus\{0\}$ such that $J_{\e}(u_{\e})=c_{\e}$ and $J'_{\e}(u_{\e})=0$.\\
As in \cite{FS}, one can prove that also $J_{0}$ has a mountain pass geometry, so we denote by $c_{V_{0}}$ the mountain pass level associated to $J_{0}$.\\
Now, let us introduce the Nehari manifold associated to (\ref{Pe}), that is
\begin{equation*}
\mathcal{N}_{\e}:= \{u\in H^{s}_{\e} \setminus \{0\} : \langle J_{\e}'(u), u \rangle =0\}.
\end{equation*}
It is easy to check that there exists $r>0$ such that $\|u\|_{H^{s}_{\e}}\geq r$ for all $u\in \mathcal{N}_{\e}$ and $\e>0$.\\
Let us denote 
$$
H_{\e}^{+}=\{u\in H^{s}_{\e}: |\supp(u^{+})\cap \Lambda_{\e}|>0\},
$$
and $\mathbb{S}_{\e}^{+}=\mathbb{S}_{\e}\cap H_{\e}^{+}$, where $\mathbb{S}_{\e}$ is the unitary sphere in $H^{s}_{\e}$.

We note that $S_{\e}^{+}$ is a not complete $C^{1, 1}$ manifold of codimension $1$ \cite{SW}, so $H^{s}_{\e}=T_{u}\mathbb{S}_{\e}^{+}\oplus \R u$ for all $u\in T_{u}\mathbb{S}_{\e}^{+}=\{v\in H^{s}_{\e}: (u, v)_{H^{s}_{\e}}=0\}$.\\
Now, we prove the following lemma  which will be fundamental to deduce our main result.
\begin{lem}\label{lemz1}
Suppose that $V$ satisfies $(V_1)$-$(V_2)$ and $f$ verifies $(f_1)$-$(f_4)$. \\
Then, the following facts hold true:
\begin{compactenum}
\item[$(a)$] For any $u\in H_{\e}^{+}$, let $h_{u}: \R_{+} \rightarrow \R$ be defined by $h_{u}(t):= J_{\e}(tu)$. Then, there is a unique $t_{u}>0$ such that $h_{u}'(t)>0$ in $(0, t_{u})$ and $h_{u}'(t)<0$ in $(t_{u}, +\infty)$.
\item[$(b)$] There is $\tau>0$, independent on $u$, such that $t_{u}\geq \tau$ for every $u\in \mathbb{S}_{\e}^{+}$. Moreover, for each compact set $\mathcal{W}\subset \mathbb{S}_{\e}^{+}$, there is $C_{\mathcal{W}}>0$ such that $t_{u}\leq C_{\mathcal{W}}$ for every $u\in \mathcal{W}$.
\item[$(c)$] The map $\hat{m}_{\e}: H_{\e}^{+}\rightarrow \mathcal{N}_{\e}$ given by $\hat{m}_{\e}(u):=t_{u}u$ is continuous and $m_{\e}:= \hat{m}|_{\mathbb{S}_{\e}^{+}}$ is a homeomorphism between $\mathbb{S}^{+}_{\e}$ and $\mathcal{N}_{\e}$. Moreover, $m^{-1}_{\e}(u)= \frac{u}{\|u\|_{H^{s}_{\e}}}$.
\item[$(d)$] If there is a sequence $(u_{n})\subset \mathbb{S}^{+}_{\e}$ such that $dist(u_{n}, \partial \mathbb{S}_{\e}^{+})\rightarrow 0$, then $\|m_{\e}(u_{n})\|_{H^{s}_{\e}}\rightarrow \infty$ and $J_{\e}(m_{\e}(u_{n}))\rightarrow \infty$. 
\end{compactenum}
\end{lem}
\begin{proof}
By using $(g_1)$-$(g_2)$, and Theorem \ref{Sembedding}, we can see that for any $u\in H^{+}_{\e}$ and $t>0$
$$
J_{\e}(tu)\geq C_{1}t^{2}\|u\|_{H^{s}_{\e}}^{2}-C_{2}t^{q}\|u\|^{q}_{H^{s}_{\e}}.
$$ 
On the other hand, by using $(g_3)$, we get for any $u\in H^{+}_{\e}$ and $t>0$
$$
J_{\e}(t u)\leq \frac{t^{2}}{2}\|u\|^{2}_{H^{s}_{\e}}-C_{1}t^{\theta}\|u\|^{\theta}_{L^{\theta}(\Lambda_{\e})}+C_{2}|\supp(u^{+})\cap \Lambda_{\e}|.
$$

Then, by the continuity of $h_{u}$, it is easy to see that there exists $t_{u}>0$ such that $\max_{t\geq 0} h_{u}(t)=h_{u}(t_{u})$, $t_{u}u\in \mathcal{N}_{\e}$ and $h_{u}'(t_{u})=0$. The uniqueness of a such $t_{u}$, follows by the assumption $(f_4)$ and by the definition of $g$. Therefore, condition $(a)$ is verified. Concerning $(b)$, we note that for any $u\in \mathbb{S}_{\e}^{+}$, from $h'_{u}(t_{u})=0$, $(g_1)$-$(g_2)$, and Theorem \ref{Sembedding}, we get
$$
t_{u}=\int_{\R^{N}} g(\e x, t_{u} u)u\, dx\leq \e C t_{u}^{2}+C_{\e}t_{u}^{q}
$$
so there exists $\tau>0$ independent of $u$, such that $t_{u}\geq \tau$.
Now, by using $(g_3)$ and $(g_4)$, we can observe that 
\begin{equation}\label{absurd}
J_{\e}(v)=J_{\e}(v)-\frac{1}{\theta}\langle J'_{\e}(v), v\rangle\geq C\|v\|_{H^{s}_{\e}}^{2} \mbox{ for any } v\in \mathcal{N}_{\e}.
\end{equation}
Hence, if there exists $(u_{n})\subset \mathcal{W}$ such that $t_{u_{n}}\rightarrow \infty$, by the compactness of $\mathcal{W}$, it follows that $u_{n}\rightarrow w$ in $H^{s}_{\e}$, and $J_{\e}(t_{n}u_{n})\rightarrow -\infty$. Taking $v_{n}=t_{n}u_{n}\in \mathcal{N}_{\e}$ in (\ref{absurd}), we can see that $\lim_{n\rightarrow \infty} J_{\e}(t_{n}u_{n})>0$, which gives a contradiction. Regarding $(c)$, we first note that $\hat{m}_{\e}$, $m_{\e}$ and $m_{\e}^{-1}$ are well defined. In fact, in view of $(a)$, for any $u\in H^{s}_{+}$ there exists a unique $\hat{m}_{\e}(u)\in \mathcal{N}_{\e}$. Moreover, if $u\in \mathcal{N}_{\e}\setminus H^{+}_{\e}$, then from $(g_4)$, we deduce that
$$
\|u\|_{H^{s}_{\e}}^{2}=\int_{\R^{N}} g(\e x, u)u\, dx=\int_{\R^{N}\setminus \Lambda_{\e}} g(\e x, u^{+})u^{+}\, dx\leq \frac{1}{K}\|u\|_{H^{s}_{\e}}^{2}
$$
which gives a contradiction because $K>2$ and $u\neq 0$. Therefore $m_{\e}^{-1}(u)=u/\|u\|_{H^{s}_{\e}}$ is well defined.
Moreover, since
$$
m^{-1}_{\e}(m_{\e}(u))=m_{\e}^{-1}(t_{u} u)=\frac{t_{u}u}{t_{u}\|u\|_{H^{s}_{\e}}}=u \quad \forall u\in H^{+}_{\e}
$$
and
$$
m_{\e}(m^{-1}_{\e}(u))=m_{\e}\left(\frac{u}{\|u\|_{H^{s}_{\e}}}\right)=t_{\frac{u}{\|u\|_{H^{s}_{\e}}}}\frac{u}{\|u\|_{H^{s}_{\e}}}=u \quad \forall u\in \mathcal{N}_{\e},
$$
we can see that $m_{\e}$ is bijective and and $m^{-1}_{\e}$ is continuous. In order to show that $\hat{m}_{\e}: H^{+}_{\e}\rightarrow \mathcal{N}_{\e}$ is continuous, let $(u_{n})\subset H^{+}_{\e}$ and $u\in H^{+}_{\e}$ such that $u_{n}\rightarrow u$ in $H^{s}_{\e}$. By using $(b)$, there exists $t_{0}>0$ such that $t_{u_{n}}\rightarrow t_{0}$. Since $t_{u_{n}}u_{n}\in \mathcal{N}_{\e}$, we can see that
$$
t_{0}^{2}\|u\|^{2}_{H^{s}_{\e}}=\lim_{n\rightarrow \infty} t_{u_{n}}^{2}\|u_{n}\|^{2}_{H^{s}_{\e}}=\lim_{n\rightarrow \infty} \int_{\R^{N}} g(\e x, t_{u_{n}} u_{n}) t_{u_{n}} u_{n}\, dx=\int_{\R^{N}} g(\e x, t_{0} u) t_{0} u\, dx
$$
which implies that $t_{0}u\in \mathcal{N}_{0}$ and $t_{u}=t_{0}$. This ends the proof of the continuity of $\hat{m}_{\e}$.
Finally, we prove $(d)$. We proceed as in Lemma $26$ in \cite{SW}. Let $(u_{n})\subset \mathbb{S}^{+}_{\e}$ such that $dist(u_{n}, \partial \mathbb{S}_{\e}^{+})\rightarrow 0$. \\
Then 
\begin{equation}\label{SW}
u_{n}^{+}\leq |u_{n}-v| \mbox{ a.e. in } \Lambda_{\e}, \mbox{ for any } v\in \partial \mathbb{S}_{\e}^{+}.
\end{equation}
Taking into account $(g_2)$, $(g_3)$, $(g_4)$, Theorem \ref{Sembedding}, and $(\ref{SW})$, we can deduce that
\begin{align*}
\int_{\R^{N}} G(\e x, t u_{n}) \, dx&\leq \int_{\Lambda_{\e}} F(t u_{n})\, dx+\frac{t^{2}}{K}\int_{\R^{N}\setminus \Lambda_{\e}} V(\e x) u_{n}^{2}\, dx \\
&\leq C_{1} t^{2} \|u_{n}^{+}\|^{2}_{L^{2}(\Lambda_{\e})}+C_{2}t^{q}  \|u_{n}^{+}\|^{q}_{L^{q}(\Lambda_{\e})}+\frac{t^{2}}{K}\|u_{n}\|^{2}_{H^{s}_{\e}} \\
&\leq C t^{2} dist(u_{n}, \partial \mathbb{S}_{\e}^{+})^{2}+C_{q} t^{q} dist(u_{n}, \partial \mathbb{S}_{\e}^{+})^{q}+\frac{t^{2}}{K}
\end{align*}
which gives
$$
\limsup_{n\rightarrow \infty} G(\e x, t u_{n})\, dx\leq \frac{t^{2}}{K}
$$
for any $t>0$.
This and the definition of $m_{\e}$, yield for any $t>0$
\begin{align}
\frac{1}{2}\liminf_{n\rightarrow \infty} \|m_{\e}(u_{n})\|^{2}_{H^{s}_{\e}}\geq \liminf_{n\rightarrow \infty} J_{\e}(m_{\e}(u_{n}))\geq \liminf_{n\rightarrow \infty} J_{\e}(t u_{n})\geq \liminf_{n\rightarrow \infty} \frac{t^{2}}{2} \|u_{n}\|^{2}_{H^{s}_{\e}}-\frac{t^{2}}{K}\geq t^{2}\left(\frac{1}{2}-\frac{1}{K}\right)
\end{align}
which implies that $\|m_{\e}(u_{n})\|_{H^{s}_{\e}}\rightarrow \infty$ and $J_{\e}(m_{\e}(u_{n}))\rightarrow \infty$, being $K>2$.

\end{proof}

\noindent
Let us define the maps $\hat{\psi}_{\e}: H^{+}_{\e} \rightarrow \R$ by $\hat{\psi}(u):= J_{\e}(\hat{m}_{\e}(u))$, and $\psi:=\hat{\psi}|_{\mathbb{S}_{\e}^{+}}$. 
The next result is a consequence of Lemma \ref{lemz1}. For more details, see \cite{SW}.
\begin{prop}\label{propz2}
Suppose that $V$ satisfies $(V_1)$-$(V_2)$ and $f$ verifies $(f_1)$-$(f_4)$.. Then, one has:
\begin{compactenum}
\item[$(a)$] $\hat{\psi}_{\e}\in C^{1}(H^{+}_{\e}, \R)$ and
\begin{equation*}
\langle \hat{\psi}_{\e}'(u), v \rangle=\frac{\|\hat{m}_{\e}(u)\|_{H^{s}_{\e}}}{\|u\|_{H^{s}_{\e}}} \langle J_{\e}'(\hat{m}_{\e}(u)), v \rangle \,,
\end{equation*}
for every $u\in H^{+}_{\e}$ and $v\in H^{s}_{\e}$;
\item[$(b)$] $\psi \in C^{1}(\mathbb{S}^{+}_{\e}, \R)$ and $\langle \psi'_{\e}(u), v \rangle = \|m_{\e}(u)\|_{H^{s}_{\e}} \langle J_{\e}'(m_{\e}(u)), v \rangle$, for every $v\in T_{u}\mathbb{S}_{\e}^{+}$.
\item[$(c)$] If $(u_{n})$ is a $(PS)_{d}$ sequence for $\psi_{\e}$, then $(m_{\e}(u_{n}))$ is a $(PS)_{d}$ sequence for $J_{\e}$. Moreover, if $(u_{n})\subset \mathcal{N}_{\e}$ is a bounded $(PS)_{d}$ sequence for $J_{\e}$, then $(m_{\e}^{-1} (u_{n}))$ is a $(PS)_{d}$ sequence for the functional $\psi_{\e}$;
\item[$(d)$] $u$ is a critical point of $\psi_{\e}$ if and only if $m_{\e}(u)$ is a nontrivial critical point for $J_{\e}$. Moreover, the corresponding critical values coincide and
\begin{equation*}
\inf_{u\in\mathbb{S}^{+}_{\e}} \psi_{\e}(u) = \inf_{u\in\mathcal{N}_{\e}} J_{\e}(u).
\end{equation*}
\end{compactenum}
\end{prop}

\noindent
Finally, we prove the following result:
\begin{lem}\label{lemma2.10}
The functional $\psi_{\e}$ satisfies the $(PS)_{c}$ on $\mathbb{S}_{\e}^{+}$ for any $c\in \R$.
\end{lem}
\begin{proof}
Let $(u_{n})\subset \mathbb{S}_{\e}^{+}$ be a $(PS)_{c}$ sequence for $\psi_{\e}$. Then, by using Proposition \ref{propz2}-$(c)$, we can infer that $(m_{\e}(u_{n}))$ is a $(PS)_{c}$ sequence for $J_{\e}$. In view of Lemma \ref{PSc}, we can see that, up to a subsequence, $m_{\e}(u_{n})$ converges strongly in $H^{s}_{\e}$. The end of lemma follows by Lemma \ref{lemz1}-$(c)$.

\end{proof}

\begin{remark}
As in \cite{SW}, we have the following characterization of $J_{\e}$ on $\mathcal{N}_{\e}$:
$$
c_{\e}= \inf_{u\in\mathcal{N}_{\e}} J_{\e}(u)=\inf_{u\in H^{+}_{\e}} \max_{t>0} J_{\e}(t u)=\inf_{u\in \mathbb{S}^{+}_{\e}} \max_{t>0} J_{\e}(t u).
$$
\end{remark}
\begin{remark}\label{rem2}
With suitable modifications, the proofs of Lemma $\ref{lemz1}$ and Proposition $\ref{propz2}$ hold when we take $\e=0$, that is when we replace $J_{\e}$ and $\mathcal{N}_{\e}$ by $J_{0}$ and $\mathcal{N}_{0}$.
Therefore, we can deduce that 
$$
c_{V_{0}}= \inf_{u\in\mathcal{N}_{0}} J_{0}(u)=\inf_{u\in H^{+}_{0}} \max_{t>0} J_{0}(t u)=\inf_{u\in \mathbb{S}^{+}_{0}} \max_{t>0} J_{0}(t u),
$$
where $\mathcal{N}_{0}=\{u\in H^{s}_{0}\setminus\{0\}: \|u\|^{2}_{H^{s}_{0}}=\int_{\R^{N}} f(u)u\, dx\}$, $H^{+}_{0}=\{u\in H^{s}_{0}: |\supp(u^{+})|>0\}$, $\mathbb{S}_{0}^{+}=\mathbb{S}_{0}\cap H^{+}_{0}$, and $\mathbb{S}_{0}$ is the unit sphere in $H^{s}_{0}$.
\end{remark}

\section{Proof of Theorem \ref{thmf}}
In order to study the multiplicity of solutions to (\ref{P}), we need introduce some useful tools.\\
Fix $\delta>0$ such that $M_{\delta}\subset \Lambda$, where
$$
M_{\delta}=\{x\in \R^{N}: dist(x, M)\leq \delta\}.
$$
Let $w$ be a ground state solution for $J_{0}$ (see for instance \cite{A4, CW}), and, for any $z\in M$ we define
$$
\Psi_{\e, z}(x)=\eta(|\e x-z|) w\left(\frac{\e x-z}{\e}\right)
$$
with $\eta\in C^{\infty}_{0}(\R_{+}, [0, 1])$ satisfying $\eta(t)=1$ if $0\leq t\leq \frac{\delta}{2}$ and $\eta(t)=0$ if $t\geq \delta$.
Then, let $t_{\e}>0$ the unique positive number such that 
$$
\max_{t\geq 0} J_{\e}(t \Psi_{\e, z})=J_{\e}(t_{\e} \Psi_{\e, z}).
$$
Finally, we consider $\Phi_{\e}(z)=t_{\e} \Psi_{\e, z}$.

\begin{lem}\label{lemma3.4}
The functional $\Phi_{\e}$ satisfies the following limit
$$
\lim_{\e\rightarrow 0} J_{\e}(\Phi_{\e}(z))=c_{V_{0}} \mbox{ uniformly in } z\in M.
$$
\end{lem}
\begin{proof}
Assume by contradiction that there there exists $\delta_{0}>0$, $(z_{n})\subset M$ and $\e_{n}\rightarrow 0$ such that 
\begin{equation}\label{4.41}
|J_{\e_{n}}(\Phi_{\e_{n}}(z_{n}))-c_{V_{0}}|\geq \delta_{0}.
\end{equation}
We first show that $\lim_{n\rightarrow \infty}t_{\e_{n}}<\infty$.
Let us observe that by using the change of variable $y=\frac{\e_{n}x-z_{n}}{\e_{n}}$, if $y\in B_{\frac{\delta}{\e_{n}}}(0)$, it follows that $\e_{n} y\in B_{\delta}(0)$ and $\e_{n} y+z_{n}\in B_{\delta}(z_{n})\subset M_{\delta}\subset \Lambda$. Since $G=F$ on $\Lambda$, we can see that 
\begin{align}
J_{\e_{n}}(\Phi_{\e_{n}}(z_{n}))&=\frac{t_{\e_{n}}^{2}}{2}\int_{\R^{N}} |(-\Delta)^{\frac{s}{2}}(\eta(|\e_{n} y|)w(y))|^{2}\, dy+\frac{t_{\e_{n}}^{2}}{2}\int_{\R^{N}} V(\e_{n} y+z_{n}) (\eta(|\e_{n} y|) w(y))^{2}\, dy \nonumber\\
&-\int_{\R^{N}} F(t_{\e_{n}}\eta(|\e_{n} y|)w(y)) \, dy.
\end{align}
Now, let assume that $t_{\e_{n}}\rightarrow \infty$. From the definition of $t_{\e_{n}}$, we get
\begin{equation}\label{3.9}
\|\Psi_{\e_{n}, z_{n}}\|^{2}_{H^{s}_{\e_{n}}}=\int_{\R^{N}} \frac{f(t_{\e_{n}}\eta(|\e_{n} y|)w(y)) (\eta(|\e_{n} y|)w(y))^{2}}{t_{\e_{n}}\eta(|\e_{n} y|)w(y)} \, dy
\end{equation}
Since $\eta=1$ in $B_{\frac{\delta}{2}}(0)$ and $B_{\frac{\delta}{2}}(0)\subset B_{\frac{\delta}{2\e_{n}}}(0)$ for $n$ big enough, and by using $(f_4)$, we obtain
\begin{align}\label{3.10}
\|\Psi_{\e_{n}, z_{n}}\|^{2}_{H^{s}_{\e_{n}}}\geq \int_{B_{\frac{\delta}{2}}(0)} \frac{f(t_{\e_{n}}w(y)) w^{2}(y)}{t_{\e_{n}}w(y)} \, dy \nonumber\\
\geq \frac{f(t_{\e_{n}}w(\bar{x}))}{t_{\e_{n}}w(\bar{x})} \int_{B_{\frac{\delta}{2}}(0)}  w^{2}(y)\, dy
\end{align}
where $w(\bar{x})=\min_{B_{\frac{\delta}{2}}(0)} w(x)$. Taking the limit as $n\rightarrow \infty$ in (\ref{3.10}), and by using $(f_3)$, we can deduce that
$$
\lim_{n\rightarrow \infty}\|\Psi_{\e_{n}, z_{n}}\|^{2}_{H^{s}_{\e_{n}}}=\infty
$$
which is a contradiction because 
$$
\lim_{n\rightarrow \infty}\|\Psi_{\e_{n}, z_{n}}\|^{2}_{H^{s}_{\e_{n}}}=\|w\|^{2}_{H^{s}_{0}}
$$
in view of dominated convergence theorem.\\
Thus, $(t_{\e_{n}})$ is bounded, and we can assume that $t_{\e_{n}}\rightarrow t_{0}\geq 0$. Clearly, if $t_{0}=0$, from $(f_1)$, by limitation of $\|\Psi_{\e_{n}, z_{n}}\|^{2}_{H^{s}_{\e_{n}}}$ and (\ref{3.9}), we can deduce that $\|\Psi_{\e_{n}, z_{n}}\|^{2}_{H^{s}_{\e_{n}}}\rightarrow 0$, which is impossible. Hence, $t_{0}>0$.

Now, by using the dominated convergence Theorem, we can see that 
$$
\lim_{n\rightarrow \infty}\int_{\R^{N}} f(\Psi_{\e_{n}, z_{n}})\Psi_{\e_{n}, z_{n}}\, dx=\int_{\R^{N}} f(w)w\, dx.
$$
Then, taking the limit as $n\rightarrow \infty$ in (\ref{3.9}),  we obtain
$$
\|w\|^{2}_{H^{s}_{0}}=\int_{\R^{N}} \frac{f(t_{0}w)}{t_{0}} w\, dx.
$$ 
By using the fact that $w\in \mathcal{N}_{0}$, we deduce that $t_{0}=1$. Moreover,
$$
\lim_{n\rightarrow \infty} J_{\e_{n}}(\Phi_{\e_{n}}(z_{n}))=J_{0}(w)=c_{V_{0}},
$$
which contradicts (\ref{4.41}).

\end{proof}

\noindent
For any $\delta>0$, let $\rho=\rho(\delta)>0$ be such that $M_{\delta}\subset B_{\delta}(0)$.
We define $\Upsilon: \R^{N}\rightarrow \R^{N}$ by setting $\Upsilon(x)=x$ for $|x|\leq \rho$ and $\Upsilon(x)=\frac{\rho x}{|x|}$ for $|x|\geq \rho$.
Then we define the barycenter map $\beta_{\e}: \mathcal{N}_{\e}\rightarrow \R^{N}$ given by
$$
\beta_{\e}(u)=\frac{\int_{\R^{N}} \Upsilon(\e x) u^{2}(x)\, dx}{\int_{\R^{N}} u^{2}(x) \,dx}.
$$

\noindent
We recall the following property for the barycenter map holds. This result is well-known in literature (see for instance \cite{FS}), however, for the sake of readers, we show it here. 
\begin{lem}\label{lemma3.5}
The function $\beta_{\e}$ verifies the following limit
$$
\lim_{\e \rightarrow 0} \beta_{\e}(\Phi_{\e}(z))=z \mbox{ uniformly in } z\in M.
$$
\end{lem}
\begin{proof}
Suppose by contradiction that there exists $\delta_{0}>0$, $(z_{n})\subset M$ and $\e_{n}\rightarrow 0$ such that 
\begin{equation}\label{4.4}
|\beta_{\e_{n}}(\Phi_{\e_{n}}(z_{n}))-z_{n}|\geq \delta_{0}.
\end{equation}
By using the definitions of $\Phi_{\e_{n}}(z_{n})$, $\beta_{\e_{n}}$ and $\eta$, we can see that 
$$
\beta_{\e_{n}}(\Psi_{\e_{n}}(z_{n}))=z_{n}+\frac{\int_{\R^{N}}[\Upsilon(\e_{n}y+z_{n})-z_{n}] |\eta(|\e_{n}y|) w(y)|^{2} \, dy}{\int_{\R^{N}} |\eta(|\e_{n}y|) w(y)|^{2}\, dy}.
$$
Taking into account $(z_{n})\subset M\subset B_{\rho}$ and the Dominated Convergence Theorem, we can infer that 
$$
|\beta_{\e_{n}}(\Phi_{\e_{n}}(z_{n}))-z_{n}|=o_{n}(1)
$$
which contradicts (\ref{4.4}).

\end{proof}

\noindent
Now, we state the following useful result whose proof can be found in \cite{FS}.
\begin{lem}\label{FS}
Let $(u_{n})\subset \mathcal{N}_{0}$ be a sequence satisfying $J_{0}(u_{n})\rightarrow c_{V_{0}}$. Then, up to subsequences, the following alternatives holds:
\begin{compactenum}[(i)]
\item $(u_{n})$ strongly converges in $H^{s}(\R^{N})$, 
\item there exists a sequence $(\tilde{z}_{n})\subset \R^{N}$ such that,  up to a subsequence, $v_{n}(x)=u_{n}(x+\tilde{z}_{n})$ converges strongly in $H^{s}(\R^{N})$.
\end{compactenum}
In particular, there exists a minimizer $w\in H^{s}(\R^{N})$ for $J_{0}$ with $J_{0}(w)=c_{V_{0}}$.
\end{lem}
At this point, we prove the following compactness result which will be crucial in the sequel.
\begin{lem}\label{prop3.3}
Let $\e_{n}\rightarrow 0$ and $(u_{n})\subset \mathcal{N}_{\e_{n}}$ be such that $J_{\e_{n}}(u_{n})\rightarrow c_{V_{0}}$. Then there exists $(\tilde{z}_{n})\subset \R^{N}$ such that $v_{n}(x)=u_{n}(x+\tilde{z}_{n})$ has a convergent subsequence in $H^{s}(\R^{N})$. Moreover, up to a subsequence, $z_{n}=\e_{n} \tilde{z}_{n}\rightarrow z_{0}\in M$.
\end{lem}
\begin{proof}
Since $\langle J'_{\e_{n}}(u_{n}), u_{n}\rangle=0$ and $J_{\e_{n}}(u_{n})\rightarrow c_{V_{0}}$, it is easy to prove that $(u_{n})$ is bounded in $H^{s}_{\e_{n}}$. 
Indeed, by using $(g3)$, $(g4)$ and $\theta>2$, we get
\begin{align*}
c_{V_{0}}+o_{n}(1)&= J_{\e_{n}}(u_{n})-\frac{1}{\theta} \langle J'_{\e_{n}}(u_{n}), u_{n}\rangle \\
&\geq \left(\frac{1}{2}-\frac{1}{\theta}\right)\|u_{n}\|^{2}_{H^{s}_{\e_{n}}}+\frac{1}{\theta}\int_{\R^{N}\setminus \Lambda_{\e_{n}}} \left[g(\e x, u_{n})u_{n}-\theta G(\e x, u_{n})\right] \, dx \\
&\geq  \left(\frac{1}{2}-\frac{1}{\theta}\right) \left(1-\frac{1}{K}\right)\|u_{n}\|^{2}_{H^{s}_{\e_{n}}}.
\end{align*}
Then, there exists $C>0$ (independent of $n$) such that $\|u_{n}\|_{H^{s}_{\e_{n}}}\leq C$ for all $n\in \N$.\\

Now, we show that there exist a sequence $(\tilde{z}_{n})\subset \R^{N}$, and constants $R>0$ and $\gamma>0$ such that
\begin{equation}\label{sacchi}
\liminf_{n\rightarrow \infty}\int_{B_{R}(\tilde{z}_{n})} |u_{n}|^{2} \, dx\geq \gamma>0.
\end{equation}
Suppose that condition \eqref{sacchi} does not hold. Then, for all $R>0$, we have
$$
\lim_{n\rightarrow \infty}\sup_{z\in \R^{N}}\int_{B_{R}(z)} |u_{n}|^{2} \, dx=0.
$$
Since we know that $(u_{n})$ is bounded in $H^{s}_{0}(\R^{N})$, we can use Lemma \ref{lions lemma} to deduce that $u_{n}\rightarrow 0$ in $L^{q}(\R^{N})$ for any $q\in (2, 2^{*}_{s})$. 
Moreover, by using $|g(\e_{n} x, t)|\leq \e_{n} |t|+C_{\e_{n}}|t|^{q-1}$ in $\R^{N}\times \R$ and $\e_{n}\rightarrow 0$, we  also obtain
\begin{align}\label{glimiti}
\lim_{n\rightarrow \infty}\int_{\R^{N}} g(\e_{n} x, u_{n}) u_{n} \,dx=0= \lim_{n\rightarrow \infty}\int_{\R^{N}} G(\e_{n} x, u_{n}) \, dx.
\end{align}
Taking into account $\langle J'_{\e_{n}}(u_{n}), u_{n}\rangle=0$ and \eqref{glimiti}, we can  infer that $\|u_{n}\|_{H^{s}_{\e_{n}}}\rightarrow 0$ as $n\rightarrow \infty$. This fact and \eqref{glimiti} imply that $J_{\e_{n}}(u_{n})\rightarrow 0$ as $n\rightarrow \infty$, which is a contradiction because of $J_{\e_{n}}(u_{n})\rightarrow c_{V_{0}}>0$. 
Now, we set $v_{n}(x)=u_{n}(x+\tilde{z}_{n})$. Then, $(v_{n})$ is bounded in $H^{s}_{0}$, and we may assume that 
$v_{n}\rightharpoonup v\not\equiv 0$ in $H^{s}_{0}$  as $n\rightarrow \infty$.
Fix $t_{n}>0$ such that $\tilde{v}_{n}=t_{n} v_{n}\in \mathcal{N}_{0}$. Since $u_{n}\in \mathcal{N}_{\e_{n}}$, we can see that 
$$
c_{V_{0}}\leq J_{0}(\tilde{v}_{n})\leq J_{\e_{n}}(t_{n}u_{n})\leq J_{\e_{n}}(u_{n})= c_{V_{0}}+o_{n}(1)
$$
which gives $J_{0}(\tilde{v}_{n})\rightarrow c_{V_{0}}$. In particular, we get $\tilde{v}_{n}\rightharpoonup \tilde{v}$ in $H^{s}_{0}$ and $t_{n}\rightarrow t^{*}>0$. Then, from the uniqueness of the weak limit, we have $\tilde{v}=t^{*}v\not\equiv 0$. \\
Now, we aim to show that 
\begin{equation}\label{elena}
\tilde{v}_{n}\rightarrow \tilde{v} \mbox{ in } H^{s}_{0}.
\end{equation} 
Since $(\tilde{v}_{n})\subset \mathcal{N}_{0}$ and $J_{0}(\tilde{v}_{n})\rightarrow c_{V_{0}}$, we can apply Lemma \ref{lemz1}-(c) and Proposition \ref{propz2}-(d) with $\e=0$ (see Remark \ref{rem2}) to infer that
$$
w_{n}=m_{0}^{-1}(\tilde{v}_{n})=\frac{\tilde{v}_{n}}{\|\tilde{v}_{n}\|_{H^{s}_{0}}}\in \mathbb{S}_{0}^{+}
$$
and
$$
\psi_{0}(w_{n})=J_{0}(\tilde{v}_{n})\rightarrow c_{V_{0}}=\inf_{v\in \mathbb{S}_{0}^{+}}\psi_{0}(v).
$$
Let us introduce the following map $\mathcal{L}: \overline{\mathbb{S}}_{0}^{+}\rightarrow \R\cup \{\infty\}$ defined by setting
$$
\mathcal{L}(u):=
\begin{cases}
\hat{\psi}_{0}(u)& \text{ if $u\in \mathbb{S}_{0}^{+}$} \\
\infty   & \text{ if $u\in \partial \mathbb{S}_{0}^{+}$}.
\end{cases}
$$ 
We note that 
\begin{itemize}
\item $(\overline{\mathbb{S}}_{0}^{+}, d_{0})$, where $d_{0}(u, v)=\|u-v\|_{H^{s}_{0}}$, is a complete metric space;
\item $\mathcal{L}\in C(\overline{\mathbb{S}}_{0}^{+}, \R\cup \{\infty\})$, by Lemma \ref{lemz1}-(d) with $\e=0$;
\item $\mathcal{L}$ is bounded below, by Proposition \ref{propz2}-(d) with $\e=0$.
\end{itemize}
Hence, by using the Ekeland's variational principle \cite{Ekeland}, we can find $(\tilde{w}_{n})\subset \mathbb{S}_{0}^{+}$ such that $(\tilde{w}_{n})$ is a $(PS)_{c_{V_{0}}}$ sequence for $\psi_{0}$ on $\mathbb{S}_{0}^{+}$ and $\|\tilde{w}_{n}-w_{n}\|_{H^{s}_{0}}=o_{n}(1)$.
From Proposition \ref{propz2}-(c) with $\e=0$, we can deduce that $m_{0}(\tilde{w}_{n})$ is a $(PS)_{c_{V_{0}}}$ sequence of $J_{0}$. 
By applying Lemma \ref{FS}, it follows that there exists $\tilde{w}\in \mathbb{S}_{0}^{+}$ such that $m_{0}(\tilde{w}_{n})\rightarrow m_{0}(\tilde{w})$ in $H^{s}_{0}$. This fact, together with Lemma \ref{lemz1}-(c) with $\e=0$, and $\|\tilde{w}_{n}-w_{n}\|_{H^{s}_{0}}=o_{n}(1)$, allow us to conclude that  $\tilde{v}_{n}\rightarrow \tilde{v}$ in $H^{s}_{0}$, that is (\ref{elena}) holds.
As a consequence, $v_{n}\rightarrow v$ in $H^{s}_{0}$ as $n\rightarrow \infty$.

In order to complete the proof of the lemma, we consider $z_{n}=\e_{n}\tilde{z}_{n}$. Our claim is to show that $(z_{n})$ admits a subsequence, still denoted by $z_{n}$, such that $z_{n}\rightarrow z_{0}$, for some $z_{0}\in M$. Firstly, we prove that $(z_{n})$ is bounded. We argue by contradiction, and we assume that, up to a subsequence, $|z_{n}|\rightarrow \infty$ as $n\rightarrow \infty$. Fixed $R>0$ such that $\Lambda \subset B_{R}(0)$, we can see that 
\begin{align}\label{pasq}
[v_{n}]^{2}+\int_{\R^{N}} V_{0} v_{n}^{2}\, dx &\leq \int_{\R^{N}} g(\e_{n} x+z_{n}, v_{n}) v_{n} \, dx \nonumber\\
&\leq \int_{B_{\frac{R}{\e_{n}}}(0)} \tilde{f}(v_{n}) v_{n} \, dx+\int_{\R^{N}\setminus B_{\frac{R}{\e_{n}}}(0)} f(v_{n}) v_{n} \, dx.
\end{align}
Then, by using the fact that $v_{n}\rightarrow v$ in $H^{s}_{0}$ as $n\rightarrow \infty$ and that $\tilde{f}(t)\leq \frac{V_{0}}{K}t$, we can see that (\ref{pasq}) implies that
$$
[v_{n}]^{2}+\int_{\R^{N}} V_{0} v_{n}^{2}\, dx=o_{n}(1),
$$
that is $v_{n}\rightarrow 0$ in $H^{s}_{0}$, which is a contradiction. Therefore, $(z_{n})$ is bounded, and we may assume that $z_{n}\rightarrow z_{0}\in \R^{N}$. Clearly, if $z_{0}\notin \overline{\Lambda}$, then we can argue as before and we deduce that $v_{n}\rightarrow 0$ in $H^{s}_{0}$, which is impossible. Hence $z_{0}\in \overline{\Lambda}$. Now, we note that if $V(z_{0})=V_{0}$, then we can infer that $z_{0}\notin \partial \Lambda$ in view of $(V_2)$, and then $z_{0}\in M$. Therefore, in the next step, we show that $V(z_{0})=V_{0}$. Suppose by contradiction that $V(z_{0})>V_{0}$.
Then, by using (\ref{elena}) and Fatou's Lemma, we get 
\begin{align*}
c_{V_{0}}=J_{0}(\tilde{v})&<\liminf_{n\rightarrow \infty}\left[\frac{1}{2}[\tilde{v}_{n}]^{2}+\frac{1}{2}\int_{\R^{N}} V(\e_{n}x+z_{n}) \tilde{v}_{n}^{2} \, dx-\int_{\R^{N}} F(\tilde{v}_{n})\, dx  \right] \\
&=\liminf_{n\rightarrow \infty}\left[\frac{t_{n}^{2}}{2}[u_{n}]^{2}+\frac{t_{n}^{2}}{2}\int_{\R^{N}} V(\e_{n}x) u_{n}^{2} \, dx-\int_{\R^{N}} F(t_{n} u_{n})\, dx  \right] \\
&\leq \liminf_{n\rightarrow \infty} J_{\e_{n}}(t_{n} u_{n}) \leq \liminf_{n\rightarrow \infty} J_{\e_{n}}(u_{n})=c_{V_{0}}
\end{align*}
which gives a contradiction.

\end{proof}

\noindent
Now, we introduce a subset $\tilde{\mathcal{N}}_{\e}$ of $\mathcal{N}_{\e}$ by taking a function $h:\R_{+}\rightarrow \R_{+}$ such that $h(\e)\rightarrow 0$ as $\e \rightarrow 0$, and setting
$$
\tilde{\mathcal{N}}_{\e}=\{u\in \mathcal{N}_{\e}: J_{\e}(u)\leq c_{V_{0}}+h(\e)\}.
$$
Fixed $z\in M$, we conclude from Lemma \ref{lemma3.4} that $h(\e)=|J_{\e}(\Phi_{\e}(z))-c_{V_{0}}|\rightarrow 0$ as $\e \rightarrow 0$. Hence $\Phi_{\e}(z)\in \tilde{\mathcal{N}}_{\e}$, and $\tilde{\mathcal{N}}_{\e}\neq \emptyset$ for any $\e>0$. Moreover, we have the following lemma.
\begin{lem}\label{lemma3.7}
$$
\lim_{\e \rightarrow 0} \sup_{u\in \tilde{\mathcal{N}}_{\e}} dist(\beta_{\e}(u), M_{\delta})=0.
$$
\end{lem}
\begin{proof}
Let $\e_{n}\rightarrow 0$ as $n\rightarrow \infty$. For any $n\in \N$, there exists $u_{n}\in \tilde{\mathcal{N}}_{\e_{n}}$ such that
$$
\sup_{u\in \tilde{\mathcal{N}}_{\e_{n}}} \inf_{z\in M_{\delta}}|\beta_{\e_{n}}(u)-z|=\inf_{z\in M_{\delta}}|\beta_{\e_{n}}(u_{n})-z|+o_{n}(1).
$$
Therefore, it is suffices to prove that there exists $(z_{n})\subset M_{\delta}$ such that 
\begin{equation}\label{3.13}
\lim_{n\rightarrow \infty} |\beta_{\e_{n}}(u_{n})-z_{n}|=0.
\end{equation}
We note that $(u_{n})\subset  \tilde{\mathcal{N}}_{\e_{n}}\subset  \mathcal{N}_{\e_{n}}$, from which we deuce that
$$
c_{V_{0}}\leq c_{\e_{n}}\leq J_{\e_{n}}(u_{n})\leq c_{V_{0}}+h(\e_{n}).
$$
This yields $J_{\e_{n}}(u_{n})\rightarrow c_{V_{0}}$. By using Lemma \ref{prop3.3}, there exists $(\tilde{z}_{n})\subset \R^{N}$ such that $z_{n}=\e_{n}\tilde{z}_{n}\in M_{\delta}$ for $n$ sufficiently large. By setting $v_{n}=u_{n}(\cdot+\tilde{z}_{n})$, we can see that
$$
\beta_{\e_{n}}(u_{n})=z_{n}+\frac{\int_{\R^{N}}[\Upsilon(\e_{n}y+z_{n})-z_{n}] v_{n}^{2} \, dy}{\int_{\R^{N}} v^{2}_{n}\, dy}.
$$
Since $\e_{n} x+z_{n}\rightarrow z_{0}\in M$, we deduce that $\beta_{\e_{n}}(u_{n})=z_{n}+o_{n}(1)$, that is (\ref{3.13}) holds.

\end{proof}

\noindent
At this point, we give the following multiplicity result by using the abstract theory in \cite{SW}. We note that, because of $f$ is not $C^{1}$, we cannot use the methods in \cite{FS}.
\begin{thm}\label{teorema}
Assume that $(V_1)$-$(V_2)$ and $(f_1)$-$(f_4)$ hold. Then, for any $\delta>0$ there exists $\bar{\e}_\delta>0$ such that, for any $\e \in (0, \bar{\e}_\delta)$, problem $(\ref{Pe})$ has at least $cat_{M_{\delta}}(M)$ positive solutions.   
\end{thm}

\begin{proof}
For any $\e>0$, we define $\alpha_\e : M \rightarrow \mathbb{S}_{\e}^{+}$ by setting $\alpha_\e(z)= m_{\e}^{-1}(\Phi_{\e}(z))$. By using Lemma \ref{lemma3.4} and Proposition \ref{propz2}-(d), we can see that
\begin{equation*}
\lim_{\e \rightarrow 0} \psi_{\e}(\alpha_{\e}(z)) = \lim_{\e \rightarrow 0} J_{\e}(\Phi_{\e}(z))= c_{V_{0}} \mbox{ uniformly in } z\in M. 
\end{equation*}  
Then, there exists $\bar{\e}>0$ such that $\{ w\in \mathbb{S}_{\e}^{+} : \psi_{\e}(w) \leq c_{V_{0}} + h(\e)\} \neq \emptyset$ for all $\e \in (0, \bar{\e})$. \\
Taking into account Lemma \ref{lemma3.4}, Lemma \ref{lemz1}-(c), Lemma \ref{lemma3.5} and Lemma \ref{lemma3.7}, we can find $\bar{\e}= \bar{\e}_{\delta}>0$ such that the following diagram
\begin{equation*}
M\stackrel{\Phi_{\e}}{\rightarrow} \Phi_{\e}(M) \stackrel{m_{\e}^{-1}}{\rightarrow} \alpha_{\e}(M)\stackrel{m_{\e}}{\rightarrow} \Phi_{\e}(M) \stackrel{\beta_{\e}}{\rightarrow} M_{\delta}
\end{equation*}    
is well defined for any $\e \in (0, \bar{\e})$. \\
By using Lemma \ref{lemma3.5}, there exists a function $\theta(\e, z)$ with $|\theta(\e, z)|<\frac{\delta}{2}$ uniformly in $z\in M$ for all $\e \in (0, \bar{\e})$ such that $\beta_{\e}(\Phi_{\e}(z))= z+ \theta(\e, z)$ for all $z\in M$. Then, we can see that $H(t, z)= z+ (1-t)\theta(\e, z)$ with $(t, z)\in [0,1]\times M$ is a homotopy between $\beta_{\e} \circ \Phi_{\e}$ and the inclusion map $id: M \rightarrow M_{\delta}$, which implies that $cat_{\alpha_{\e}(M)} \alpha_{\e}(M)\geq cat_{M_{\delta}}(M)$. Hence, by using Lemma \ref{lemma2.10} and Corollary $28$ in \cite{SW}, with $c= c_{\e}\leq c_{V_{0}}+h(\e) =d$ and $K= \alpha_{\e}(M)$, we can see that $\Psi_{\e}$ has at least $cat_{\alpha_{\e}(M)} \alpha_{\e}(M)$ critical points on $\{ w\in \mathbb{S}_{\e}^{+} : \psi_{\e}(w) \leq c_{V_{0}} + h(\e)\}$. \\
In view of Proposition \ref{propz2} and $cat_{\alpha_{\e}(M)} \alpha_{\e}(M)\geq cat_{M_{\delta}}(M)$, we can infer that $J_{\e}$ admits at least $cat_{M_{\delta}}(M)$ critical points in $\tilde{\mathcal{N}}_{\e}$.       

\end{proof}

\noindent
Now, we are able to give the proof of our main result.
\begin{proof}
Take $\delta>0$ such that $M_\delta \subset \Lambda$. We begin proving that there exists $\tilde{\e}_{\delta}>0$ such that for any $\e \in (0, \tilde{\e}_{\delta})$ and any solution $u_{\e} \in \tilde{\mathcal{N}}_{\e}$ of \eqref{Pe}, it results 
\begin{equation}\label{infty}
\|u_{\e}\|_{L^{\infty}(\R^{N}\setminus \Lambda_{\e})}<a. 
\end{equation}
Assume by contradiction that there exists $\e_{n}\rightarrow 0$, $u_{\e_{n}}\in \tilde{\mathcal{N}}_{\e_{n}}$ such that $J'_{\e_{n}}(u_{\e_{n}})=0$ and $\|u_{\e_{n}}\|_{L^{\infty}(\R^{N}\setminus \Lambda_{\e_{n}})}\geq a$. 
Since $J_{\e_{n}}(u_{\e_{n}}) \leq c_{V_{0}} + h(\e_{n})$ and $h(\e_{n})\rightarrow 0$, we can argue as in the  first part of the proof of Lemma \ref{prop3.3}, to deduce that $J_{\e_{n}}(u_{\e_{n}})\rightarrow c_{V_{0}}$.
Then, by using Lemma \ref{prop3.3}, we can find $(\tilde{z}_{n})\subset \R^{N}$ such that $\e_{n}\tilde{z}_{n}\rightarrow z_{0} \in M$. \\
Now, if we choose $r>0$ such that $B_{2r}(z_{0})\subset \Lambda$, we can see $B_{\frac{r}{\e_{n}}}(\frac{z_{0}}{\e_{n}})\subset \Lambda_{\e_{n}}$. In particular, for any $z\in B_{\frac{r}{\e_{n}}}(\tilde{z}_{n})$ there holds
\begin{equation*}
\left|z - \frac{z_{0}}{\e_{n}}\right| \leq |z- \tilde{z}_{n}|+ \left|\tilde{z}_{n} - \frac{z_{0}}{\e_{n}}\right|<\frac{2r}{\e_{n}}\, \mbox{ for } n \mbox{ sufficiently large. }
\end{equation*}
Therefore $\R^{N}\setminus \Lambda_{\e_{n}}\subset \R^{N} \setminus B_{\frac{r}{\e_{n}}}(\tilde{z}_{n})$ for any $n$ big enough. By using Lemma $2.6$ in \cite{AM}, there exists $R>0$ such that $w_{n}(x)<a$ for $|x|\geq R$ and $n\in \N$, where $w_{n}(x)=u_{\e_{n}}(x+ \tilde{z}_{n})$. \\
Hence $u_{\e_{n}}(x)<a$ for any $x\in \R^{N}\setminus B_{R}(\tilde{z}_{n})$ and $n\in \N$. As a consequence, there exists $\nu \in \N$ such that for any $n\geq \nu$ and $\frac{r}{\e_{n}}>R$, it holds $\R^{N}\setminus \Lambda_{\e_{n}}\subset \R^{N} \setminus B_{\frac{r}{\e_{n}}}(\tilde{z}_{n})\subset \R^{N}\setminus B_{R}(\tilde{z}_{n})$, which gives $u_{\e_{n}}(x)<a$ for any $x\in \R^{N}\setminus \Lambda_{\e_{n}}$, and this gives a contradiction. \\
Now, let $\bar{\e}_{\delta}$ given in Theorem \ref{teorema} and take $\e_{\delta}= \min \{\tilde{\e}_{\delta}, \bar{\e}_{\delta}\}$. Fix $\e \in (0, \e_{\delta})$. By Theorem \ref{teorema}, we know that problem \eqref{Pe} admits $cat_{M_{\delta}}(M)$ nontrivial solutions $u_{\e}$. Since $u_{\e}\in \tilde{\mathcal{N}}_{\e}$ satisfies \eqref{infty}, from the definition of $g$ it follows that $u_{\e}$ is a solution of \eqref{R}. \\
Taking $v_{\e}(x)=u_{\e}(x/\e)$, we can infer that $v_{\e}$ is a solution to (\ref{P}). Finally, we prove that if $x_\e$ is a global maximum point of $v_\e(x)$, then $V(x_\e)\rightarrow V_{0}$ as $\e\rightarrow 0$. \\
Firstly, we note that there exists $\tau_{0}>0$ such that $v_{\e}(x_{\e})\geq \tau_{0}$ for all $\e\in (0, \e_{\delta})$. \\
Then, if $y_{\e}=\frac{x_{\e}-\e \tilde{z}_{\e}}{\e}$, we can see that $y_{\e}$ is a global maximum point of $w_{\e}=u_{\e}(\cdot+\tilde{z}_{\e})$ and it results that $w_{\e}(y_{\e})\geq \tau_{0}$ for all $\e\in (0, \e_{\delta})$.
Now, we argue by contradiction, and we assume that there exists a sequence $\e_{n}\rightarrow 0$ and $\gamma>0$ such that 
\begin{equation}\label{terabs}
V(x_{\e_{n}})\geq V_{0}+\gamma \mbox{ for all } n\in \N.
\end{equation}
By using Lemma $2.6$ in \cite{AM}, we know that $w_{\e_{n}}\rightarrow 0$ as $|x|\rightarrow \infty$, and uniformly in $n\in \N$. \\
Thus $(y_{\e_{n}})$ is a bounded sequence. Moreover, up to a subsequence, we know that there exists $z_{0}\in M$ such that $V(z_{0})=V_{0}$ and $\e_{n} \tilde{z}_{\e_{n}}\rightarrow z_{0}$. 
Hence, $x_{\e_{n}}=\e_{n}\tilde{z}_{\e_{n}}+\e_{n}y_{\e_{n}}\rightarrow z_{0}$ and, by the continuity of $V$, we deduce that $V(x_{\e_{n}})\rightarrow V_{0}$ as $n\rightarrow \infty$, which contradicts (\ref{terabs}).         
This concludes the proof of Theorem \ref{thmf}.

\end{proof}

\section{Multiple solutions for the supercritical problem (\ref{Pcritical})}

In this section we deal with the existence of multiple solutions for (\ref{Pcritical}).
To study this problem we follow the approaches developed in \cite{CY, FF, Rab}.\\
Firstly we truncate the nonlinearity $f(u)=|u|^{q-2}u+\lambda |u|^{r-2}u$ as follows.

Let $K>0$ be a real number, whose value will be fixed later, and we set
$$
f_{\lambda}(t):=
\begin{cases}
0 & \text{ if $t<0$} \\
t^{q-1}+\lambda t^{r-1}& \text{ if $0\leq t<K$} \\
(1+\lambda K^{r-q})t^{q-1}   & \text{ if $t \geq K$}.
\end{cases}
$$
Then, it is clear that $f_{\lambda}$ satisfies the assumptions $(f_1)$-$(f_4)$ ($(f_3)$ holds with $\theta=q>2$).

Moreover
\begin{equation}\label{fk}
f_{\lambda}(t)\leq (1+\lambda K^{r-q})t^{q-1} \mbox{ for all } t\geq 0.
\end{equation}
Therefore, we can consider the following truncated problem
\begin{equation}\label{TR}
(-\Delta)^{s} u + V(\e x)u =  f_{\lambda}(u) \mbox{ in } \R^{N}. 
\end{equation}
It is easy to see that weak solutions of (\ref{TR}) are critical points of the functional $I_{\e, \lambda}: H^{s}_{\e}\rightarrow \R$ defined by
$$
I_{\e, \lambda}(u)=\frac{1}{2} \|u\|_{H^{s}_{\e}}^{2}-\int_{\R^{N}} F_{\lambda}(u)\, dx.
$$
We also consider the autonomous functional 
$$
I_{0, \lambda}(u)=\frac{1}{2} \|u\|_{H^{s}_{0}}^{2}-\int_{\R^{N}} F_{\lambda}(u)\, dx.
$$
By using Theorem \ref{thmf}, we know that for any $\lambda\geq 0$ and $\delta>0$ there exists $\bar{\e}(\delta, \lambda)>0$ such that, for any  $\e\in (0, \bar{\e}(\delta, \lambda))$, problem (\ref{TR}) admits at least $cat_{M_{\delta}}(M)$ positive solutions $u_{\e, \lambda}$.\\
Now, we prove that it is possible to estimate the $H^{s}_{\e}$-norm of these solutions, uniformly in $\lambda$ and $\e$ sufficiently small.
\begin{lem}\label{Fig1}
There exists $\bar{C}>0$ such that $\|u_{\e, \lambda}\|_{H^{s}_{\e}}\leq \bar{C}$ for any $\e>0$ sufficiently small and uniformly in $\lambda$.
\end{lem}
\begin{proof}
From the proof of Theorem \ref{thmf}, we know that any solution $u_{\e, \lambda}$ of (\ref{TR}), satisfies the following inequality
$$
I_{\e, \lambda}(u_{\e, \lambda})\leq c_{V_{0}, \lambda}+h_{\lambda}(\e)
$$
where $c_{V_{0}, \lambda}$ is the mountain pass level related to the functional $I_{\e, \lambda}$, and $h_{\lambda}(\e)\rightarrow 0$ as $\e\rightarrow 0$.
Then, decreasing $\bar{\e}(\delta, \lambda)$ if necessary, we can suppose that 
\begin{equation}
I_{\e, \lambda}(u_{\e, \lambda})\leq c_{V_{0}, \lambda}+1
\end{equation}
for any $\e\in (0, \bar{\e}(\delta, \lambda))$. By using the fact that $c_{V_{0}, \lambda}\leq c_{V_{0}, 0}$
for any $\lambda\geq 0$, we can deduce that 
\begin{equation}\label{F1}
I_{\e, \lambda}(u_{\e, \lambda})\leq c_{V_{0}, 0}+1
\end{equation}
for any $\e\in (0, \bar{\e}(\delta, \lambda))$.

We can also note that 
\begin{align}\label{F2}
I_{\e, \lambda}(u_{\e, \lambda})&=I_{\e, \lambda}(u_{\e, \lambda})-\frac{1}{\theta}\langle I'_{\e, \lambda}(u_{\e, \lambda})\rangle \nonumber\\
&=\left(\frac{1}{2}-\frac{1}{\theta}\right)\|u_{\e, \lambda}\|^{2}_{H^{s}_{\e}}+\int_{\R^{N}} \frac{1}{\theta} f_{\lambda}(u_{\e, \lambda})u_{\e, \lambda}-F_{\lambda}(u_{\e, \lambda})\, dx \nonumber\\
&\geq  \left(\frac{1}{2}-\frac{1}{\theta}\right)\|u_{\e, \lambda}\|^{2}_{H^{s}_{\e}}
\end{align}
where in the last inequality we have used the assumption $(f_3)$.
Putting together (\ref{F1}) and (\ref{F2}), we can infer that 
$$
\|u_{\e, \lambda}\|_{H^{s}_{\e}}\leq \left[\left(\frac{2\theta}{\theta-2} \right)(c_{V_{0}, 0}+1)\right]^{\frac{1}{2}}
$$
for any $\e\in (0, \bar{\e}(\delta, \lambda))$.

\end{proof}

\noindent
Now, our claim is to prove that $u_{\e, \lambda}$ is a solution of the original problem (\ref{Pcritical}). To do this, we will show that we can find $K_{0}>0$ such that for any $K\geq K_{0}$, there exists $\lambda_{0}=\lambda_{0}(K)>0$ such that 
\begin{equation}
\|u_{\e, \lambda}\|_{L^{\infty}(\R^{N})}\leq K \mbox{ for all }  \lambda\in [0, \lambda_{0}].
\end{equation}
In order to achieve our purpose, we adapt the Moser iteration technique \cite{Moser} (see also \cite{CY, FF, Rab}) to the solutions $v_{\e, \lambda}(x, y)$ of the extended problem \cite{CS}
\begin{equation}\label{TRE}
\left\{
\begin{array}{ll}
\dive(y^{1-2s} \nabla v_{\e, \lambda})=0 & \mbox{ in } \R^{N+1}_{+} \\
\frac{\partial v_{\e, \lambda}}{\partial \nu^{1-2s}}=\kappa_{s} [-V(\e x) v_{\e, \lambda}(x, 0)+f_{\lambda}(v_{\e, \lambda}(x, 0))]  & \mbox{ on } \partial\R^{N+1}_{+} \\
\end{array},
\right.
\end{equation}
related to (\ref{TR}).
We recall that $v_{\e, \lambda}$ is a weak solution to (\ref{TRE}), if $v_{\e, \lambda}$ satisfies the following identity
\begin{equation}\label{critpoint}
\iint_{\R^{N+1}_{+}} y^{1-2s} \nabla v_{\e, \lambda} \nabla \varphi \, dx dy+ \int_{\R^{N}}V(\e x) v_{\e, \lambda}(x, 0) \varphi(x, 0) \, dx=\kappa_{s}\int_{\R^{N}}  f_{\lambda}(v_{\e, \lambda}(x, 0))\varphi(x, 0) \, dx
\end{equation}
for any $\varphi\in X^{s}$. Here $X^{s}$ is the completion of $C_{c}^{\infty}(\overline{\R^{N+1}_{+}})$ with respect to the norm
$$
\iint_{\R^{N+1}_{+}} y^{1-2s} |\nabla v|^{2} \, dx dy+ \int_{\R^{N}}V(\e x) v^{2}(x, 0) \, dx.
$$
For simplicity, we will assume that $\kappa_{s}=1$, and we set $v:=v_{\e, \lambda}$.
For any $L>0$, we define $v_{L}:=\min\{v, L\}\geq 0$, where $\beta>1$ will be chosen later, and let $w_{L}=v v_{L}^{\beta-1}$. 

Taking $\varphi:=v_{L}^{2(\beta-1)}v$ in (\ref{critpoint}), we can see that
\begin{align}\label{conto1}
\iint_{\R^{N+1}_{+}} &y^{1-2s}v^{2(\beta-1)}_{L} |\nabla v|^{2} \, dxdy +\iint_{\{v<L\}} 2(\beta-1) y^{1-2s} v^{2(\beta-1)}_{L} |\nabla v|^{2} \, dx dy   \nonumber \\
&=\int_{\R^{N}} f_{\lambda}(v(x, 0)) v_{L}^{2(\beta-1)}(x, 0) v(x, 0)  \,dx-\int_{\R^{N}} V(\e x) v^{2}(x, 0) v_{L}^{2(\beta-1)}(x, 0) \, dx.
\end{align}
Putting together (\ref{conto1}), (\ref{fk}) and $(V_1)$, we get
\begin{align}\label{conto2}
&\iint_{\R^{N+1}_{+}} y^{1-2s}v^{2(\beta-1)}_{L} |\nabla v|^{2} \, dxdy\leq C_{\lambda, K} \int_{\R^{N}} v^{q}(x, 0) v_{L}^{2(\beta-1)}(x, 0) \, dx
\end{align}
where $C_{\lambda, K}=1+\lambda K^{r-q}$.
On the other hand, from Theorem \ref{Sembedding} and $\beta>1$, we have
\begin{align}\label{conto3}\begin{split}
\|w_{L}(\cdot, 0)\|_{L^{2^{*}_{s}}(\R^{N})}^{2}&\leq S_{*} \iint_{\R^{N+1}_{+}} y^{1-2s} |\nabla w_{L}|^{2}\, dx dy \\
&=S_{*}\iint_{\R^{N+1}_{+}} y^{1-2s} |v_{L}^{\beta-1}\nabla v+(\beta-1) v v_{L}^{\beta-2} \nabla v_{L}|^{2}\, dx dy \\
&\leq 2S_{*} \left(\iint_{\R^{N+1}_{+}} y^{1-2s} (\beta-1)^{2}v_{L}^{2(\beta-1)}|\nabla v|^{2} \, dx dy+\iint_{\R^{N+1}_{+}} y^{1-2s} |v_{L}^{\beta-1} \nabla v_{L}|^{2}\, dx dy   \right)\\
&\leq 2 S_{*} ((\beta-1)^{2}+1) \iint_{\R^{N+1}_{+}} y^{1-2s} v_{L}^{2(\beta-1)}|\nabla v|^{2}\, dx dy \\
&=2 S_{*} \beta^{2} \left[\left(\frac{\beta-1}{\beta}\right)^{2}+\frac{1}{\beta}^{2}\right] \iint_{\R^{N+1}_{+}} y^{1-2s} v_{L}^{2(\beta-1)}|\nabla v|^{2}\, dx dy \\
&\leq 4 S_{*} \beta^{2} \iint_{\R^{N+1}_{+}} y^{1-2s} v_{L}^{2(\beta-1)}|\nabla v|^{2}\, dx dy.
\end{split}\end{align}
Taking into account (\ref{conto2}) and (\ref{conto3}), and by using H\"older inequality, we deduce that
\begin{align}\label{conto4}
\|w_{L}(\cdot, 0)\|_{L^{2^{*}_{s}}(\R^{N})}^{2}\leq C_{1} \beta^{2} C_{\lambda, K} \left(\int_{\R^{N}} v^{2^{*}_{s}}(x, 0)\, dx\right)^{\frac{q-2}{2^{*}_{s}}} \left(\int_{\R^{N}} w_{L}^{\frac{2 2^{*}_{s}}{2^{*}_{s}-(q-2)}}(x, 0) \, dx\right)^{\frac{2^{*}_{s}-(q-2)}{2^{*}_{s}}}
\end{align}
where  $2<\frac{2 2^{*}_{s}}{2^{*}_{s}-(q-2)}<2^{*}_{s}$ and $C_{1}>0$.
In view of Lemma \ref{Fig1}, we can see that
\begin{align}\label{conto4}
\|w_{L}(\cdot, 0)\|_{L^{2^{*}_{s}}(\R^{N})}^{2}\leq C_{2} \beta^{2} C_{\lambda, K} \bar{C}^{\frac{q-2}{2^{*}_{s}}} \|w_{L}(\cdot, 0)\|_{L^{\alpha^{*}}(\R^{N})}^{2} 
\end{align}
where 
$$
\alpha^{*}=\frac{2 2^{*}_{s}}{2^{*}_{s}-(q-2)}.
$$ 
Now, we observe that if $v^{\beta}\in L^{\alpha^{*}}(\R^{N})$, from the definition of $w_{L}$, $v_{L}\leq v$, and  (\ref{conto4}), we obtain
\begin{align}\label{conto5}
\|w_{L}(\cdot, 0)\|_{L^{2^{*}_{s}}(\R^{N})}^{2}\leq C_{3} \beta^{2} C_{\lambda, K} \bar{C}^{\frac{q-2}{2^{*}_{s}}} \left(\int_{\R^{N}} v^{\beta \alpha^{*}}(x, 0)\, dx\right)^{\frac{2}{\alpha^{*}}}<\infty.
\end{align}
By passing to the limit in (\ref{conto5}) as $L \rightarrow +\infty$, the Fatou's Lemma yields
\begin{align}\label{conto6}
\|v(\cdot, 0)\|_{L^{\beta 2^{*}_{s}}(\R^{N})}\leq (C_{4} C_{\lambda, K})^{\frac{1}{2\beta}} \beta^{\frac{1}{\beta}} \|v(\cdot, 0)\|_{L^{\beta \alpha^{*}}(\R^{N})}
\end{align}
whenever $v^{\beta \alpha^{*}}\in L^{1}(\R^{N})$.\\
Now, we set $\beta:=\frac{2^{*}_{s}}{\alpha^{*}}>1$, and we observe that, being $v\in L^{2^{*}_{s}}(\R^{N})$, the above inequality holds for this choice of $\beta$. Then, by using the fact that $\beta^{2}\alpha^{*}=\beta 2^{*}_{s}$, it follows that \eqref{conto6} holds with $\beta$ replaced by $\beta^{2}$.
Therefore, we can see that
\begin{align*}
\|v(\cdot, 0)\|_{L^{\beta^{2} 2^{*}_{s}}(\R^{N})}\leq (C_{4} C_{\lambda, K})^{\frac{1}{2\beta^2}} \beta^{\frac{2}{\beta^{2}}}  \|v(\cdot, 0)\|_{L^{\beta^{2} \alpha^{*}}(\R^{N})}\leq  (C_{4} C_{\lambda, K})^{\frac{1}{2}\left(\frac{1}{\beta}+\frac{1}{\beta^{2}}\right)} \beta^{\frac{1}{\beta^{2}}+\frac{2}{\beta^{2}}} \|v(\cdot, 0)\|_{L^{\beta \alpha^{*}}(\R^{N})}.
\end{align*}
Iterating this process, and recalling that $\beta \alpha^{*}:=2^{*}_{s}$, we can infer that for every $m\in \N$
\begin{align}\label{conto7}
 \|v(\cdot, 0)\|_{L^{\beta^{m} 2^{*}_{s}}(\R^{N})} \leq (C_{4} C_{\lambda, K})^{\sum_{j=1}^{m}\frac{1}{2\beta^{j}}} \beta^{\sum_{j=1}^{m} j\beta^{-j}} \|v(\cdot, 0)\|_{L^{2^{*}_{s}}(\R^{N})}.
\end{align}
Taking the limit in (\ref{conto7}) as $m \rightarrow +\infty$ and by using Lemma \ref{Fig1}, we get
\begin{align}\label{conto9}
\|v(\cdot, 0)\|_{L^{\infty}(\R^{N})}\leq (C_{5} C_{\lambda, K})^{\gamma_{1}} \beta^{\gamma_{2}}
\end{align}
where $C_{5}$ depends on $\bar{C}$ and $S_{*}$ (see Theorem \ref{Sembedding}), and
$$
\gamma_{1}:=\frac{1}{2}\sum_{j=1}^{\infty}\frac{1}{\beta^{j}}<\infty \quad \mbox{ and } \quad \gamma_{2}:=\sum_{j=1}^{\infty}\frac{j}{\beta^{j}}<\infty.
$$
Next, we will find some suitable value of $K$ and $\lambda$ such that the following inequality holds
$$
(C_{5} C_{\lambda, K})^{\gamma_{1}} \beta^{\gamma_{2}}\leq K,
$$
or equivalently
$$
1+\lambda K^{r-q}\leq C_{5}^{-1} \beta^{-\frac{\gamma_{2}}{\gamma_{1}}} K^{\frac{1}{\gamma_{1}}}.
$$
Take $K>0$ such that 
$$
\frac{K^{\frac{1}{\gamma_{1}}}}{C_{5}\beta^{\frac{\gamma_{2}}{\gamma_{1}}}}-1>0,
$$
and fix $\lambda_{0}>0$ such that 
$$
\lambda\leq \lambda_{0}\leq \left[\frac{K^{\frac{1}{\gamma_{1}}}}{C_{5}\beta^{\frac{\gamma_{2}}{\gamma_{1}}}}-1\right] \frac{1}{K^{r-q}}.
$$
Therefore, by using (\ref{conto9}), we can infer that 
$$
\|u_{\e, \lambda}\|_{L^{\infty}(\R^{N})}\leq K \mbox{ for all }  \lambda\in [0, \lambda_{0}],
$$
that is $u_{\e, \lambda}$ is a solution of (\ref{Pcritical}). This ends the proof of Theorem \ref{thmf2}

%\smallskip


\begin{thebibliography}{777}

\bibitem{AF}
C. O. Alves and G. M. Figueiredo,  
{\it Multiplicity of positive solutions for a quasilinear problem in $\R^{N}$ via penalization method}, 
Adv. Nonlinear Stud. {\bf 5} (2005), no. 4, 551--572. 


\bibitem{AM}
C. O. Alves and O. H. Miyagaki, 
{\it Existence and concentration of solution for a class of fractional elliptic equation in $\R^{N}$ via penalization method},
Calc. Var. Partial Differential Equations {\bf 16} (2016).

\bibitem{AR}
A. Ambrosetti  P. H. and Rabinowitz,
{\it Dual variational methods in critical point theory and applications}, 
J. Funct. Anal. {\bf 14} (1973), 349--381.


\bibitem{A1}
V. Ambrosio, 
{\it Periodic solutions for a pseudo-relativistic Schr\"{o}dinger equation}, 
Nonlinear Anal. {\bf 120} (2015), 262--284.

\bibitem{A2}
V. Ambrosio, 
{\it Ground states solutions for a non-linear equation involving a pseudo-relativistic Schr\"odinger operator}, 
J. Math. Phys. {\bf 57} (2016), no. 5, 051502, 18 pp.

\bibitem{A3}
V. Ambrosio,
{\it Ground states for superlinear fractional Schr\"odinger equations in $\R^{N}$}, 
Ann. Acad. Sci. Fenn. Math. 41 (2016), no. 2, 745--756.

\bibitem{A4}
V. Ambrosio,
{\it Multiple solutions for a nonlinear scalar field equation involving the Fractional Laplacian}, Preprint. arXiv: 1603.09538.

\bibitem{CS}
L.A. Caffarelli and L. Silvestre,
{\it An extension problem related to the fractional Laplacian},
Comm. Partial Differential Equations {\bf 32} (2007), 1245--1260.



\bibitem{CY}
J. Chabrowski and J. Yang,
{\it Existence theorems for elliptic equations involving supercritical Sobolev exponent},
Adv. Differential Equations {\bf 2} (1997), 231--256.


\bibitem{CW}
X. J. Chang and Z.Q. Wang, 
{\it Ground state of scalar field equations involving fractional Laplacian with general nonlinearity}, 
Nonlinearity {\bf 26}, 479--494 (2013).

\bibitem{DDPW}
J. D\'avila, M. del Pino, and J. Wei, 
{\it Concentrating standing waves for the fractional nonlinear Schr\"odinger equation}, 
J. Differential Equations {\bf 256} (2014), no. 2, 858--892.


\bibitem{DF}
M. Del Pino and P. L. Felmer, 
{\it Local Mountain Pass for semilinear elliptic problems in unbounded domains},
Calc. Var. Partial Differential Equations, {\bf 4} (1996), 121--137.

\bibitem{DPV}
E. Di Nezza, G. Palatucci and E. Valdinoci,
{\it Hitchhiker's guide to the fractional Sobolev spaces},
Bull. Sci. math. {\bf 136} (2012), 521--573.


\bibitem{DPPV}
S. Dipierro, G. Palatucci, and E. Valdinoci, 
{\it Existence and symmetry results for a Schr\"odinger type problem involving the fractional Laplacian}, 
Le Matematiche (Catania) {\bf 68}, 201--216 (2013).


\bibitem{Ekeland}
I. Ekeland,
{\it On the variational principle},
J. Math. Anal. Appl. {\bf 47} (1974), 324--353.	


\bibitem{FFV}
M. M. Fall, F. Mahmoudi and E. Valdinoci, 
{\it Ground states and concentration phenomena for the fractional Schr\"odinger equation},
Nonlinearity {\bf 28} (2015), no. 6, 1937--1961.

	
\bibitem{FQT}
P. Felmer, A. Quaas and J. Tan,
{\it Positive solutions of the nonlinear {S}chr{\"o}dinger equation with the fractional {L}aplacian},
Proc. Roy. Soc. Edinburgh Sect. A {\bf 142} (2012), 1237--1262.


\bibitem{FF}
G. M. Figueiredo and M. Furtado, 
{\it Positive solutions for some quasilinear equations with critical and supercritical growth},
Nonlinear Anal. {\bf 66} (2007), no. 7, 1600--1616.


\bibitem{FS}
G.M. Figueiredo and G. Siciliano,
{\it A multiplicity result via Ljusternick-Schnirelmann category and Morse theory for a fractional Schr\"odinger equation in $\R^{N}$},
NoDEA  {\bf 23} (2016), no. 2, Art. 12, 22 pp.


\bibitem{FLS}
R.L. Frank, E. Lenzmann and L. Silvestre,
{\it Uniqueness of radial solutions for the fractional Laplacian},
Commun. Pur. Appl. Math. {\bf 69} (2016), 1671--1726.

\bibitem{HZ}
X. He and W. Zou,
{\it Existence and concentration result for the fractional Schr\"odinger equations with critical nonlinearities},
Calc. Var. Partial Differential Equations {\bf 55} (2016), no. 4, Paper No. 91, 39 pp.

\bibitem{Laskin1}
N. Laskin,
{\it Fractional quantum mechanics and L\'evy path integrals},
Phys. Lett. A {\bf 268} (2000), no. 4-6, 298--305. 

\bibitem{Laskin2}
N. Laskin,
{\it Fractional Schr\"odinger equation},
Phys. Rev. E (3) {\bf 66} (2002), no. 5, 056108, 7 pp.
81Q05.


\bibitem{MBRS}
G. Molica Bisci, V. R\u{a}dulescu and R. Servadei,
{\it Variational Methods for Nonlocal Fractional Problems},
{\em Cambridge University Press}, \textbf{162} Cambridge, 2016.

\bibitem{MBR}
G. Molica Bisci and V. R\u{a}dulescu,
{\it Ground state solutions of scalar field fractional Schr\"odinger equations},
Calc. Var. Partial Differential Equations {\bf 54} (2015), no. 3, 2985--3008.

\bibitem{Moser}
J. Moser,
{\it A new proof of De Giorgi's theorem concerning the regularity problem for elliptic differential equations},
Comm. Pure Appl. Math. {\bf 13} (1960), 457--468.


\bibitem{Rab}
P.H. Rabinowitz,
{\it Variational methods for nonlinear elliptic eigenvalue problems},
Indiana Univ. Math. J. 23 (1973/74), 729--754.

\bibitem{Secchi1}
S. Secchi,
{\it Ground state solutions for nonlinear fractional Schr\"odinger equations in $\R^{N}$},
J. Math. Phys. {\bf 54} (2013), 031501.


\bibitem{SZ}
X. Shang and J. Zhang,
{\it Concentrating solutions of nonlinear fractional Schr\"odinger equation with potentials},
J. Differential Equations {\bf 258} (2015), no. 4, 1106--1128.



\bibitem{SW}
A. Szulkin and T. Weth,
{\it The method of Nehari manifold},
in Handbook of Nonconvex Analysis and Applications, edited by D. Y. Gao and D. Montreanu (International Press, Boston, 2010), pp. 597--632.




\end{thebibliography}
\end{document}